\documentclass{amsart}
\usepackage{graphicx,xcolor}

\title{Generalized Hamming weights of codes arising from complete intersections}

\author[E. Camps Moreno]{Eduardo Camps Moreno}
\author[F. Salizzoni]{Flavio Salizzoni}
\author[R. San-José]{Rodrigo San-José} 
\address[Eduardo Camps Moreno]{Université de Bordeaux, Talence, France}
\email{eduardo.camps-moreno@math.u-bordeaux.fr}
\address[Flavio Salizzoni]{Max Planck Institute for Mathematics in the Sciences, Leipzig, Germany}
\email{flavio.salizzoni@mis.mpg.de}
\address[Rodrigo San-José]{Department of Mathematics\\ Virginia Tech\\ Blacksburg, VA USA}
\email{rsanjose@vt.edu}

\thanks{The second author was supported by the P500PT-222344 SNSF project. The third author was partially supported by the NSF grant DMS-2401558, the Commonwealth Cyber Initiative, an AMS-Simons Travel Grant, and by Grant PID2022-138906NB-C21 funded by MICIU/AEI/10.13039/501100011033 and by ERDF/EU}

\usepackage{geometry}
\usepackage{amssymb,amsthm,latexsym,amsmath,amscd,graphicx,url,xcolor,comment,cite}
\usepackage[mathscr]{euscript}
\usepackage[hidelinks]{hyperref}

\theoremstyle{plain}
\newtheorem{theorem}{Theorem}[section]
\newtheorem{proposition}[theorem]{Proposition}
\newtheorem{conjecture}[theorem]{Conjecture}

\newtheorem{remark}[theorem]{Remark}
\newtheorem{corollary}[theorem]{Corollary}

\newcommand{\F}{\mathbb F}

\newcommand{\Cc}{\mathcal C}

\definecolor{violet}{rgb}{0.5, 0.0, 1.0}

\begin{document}
\begin{abstract}
We provide a positive answer to a conjecture proposed by Tohǎneanu and Van Tuyl regarding the minimum distance of codes whose underlying set of points is a reduced complete intersection. Despite the technical nature of the conjecture, we show that it follows directly from a not-well-known refinement of the classical Bézout bound for overdetermined polynomial systems. For completeness, this paper presents a self-contained proof of this refined bound. Furthermore, we show that using the same approach, it is possible to obtain a bound on the generalized Hamming weights of such a code and, more generally, to control the minimum distance of the codes obtained by evaluating forms of degree $d$ on the points of a zero-dimensional complete intersection.
\end{abstract}
\maketitle
\section{Introduction}
Let $q$ be a prime power, and $\F_q$ the finite field with $q$ elements. A linear code $\Cc$ over $\F_q$ is a linear subspace of $\F_q^n$, for some $n\geq 1$. We say that $\Cc$ is an $[n,k]$ code if $C\subset \F_q^n$ and $\dim_{\F_q}C=k$. For any subcode $\mathcal{D}\subseteq\Cc$, the support $\mathrm{supp}(\mathcal{D})$ of $\mathcal{D}$ is the set of coordinate positions where at least one codeword in $\mathcal{D}$ has a non-zero entry, that is
$$\mathrm{supp}(\mathcal{D})=\{i:\exists c\in\mathcal{D}\text{ such that }c_i\neq0\}.$$
The $r$-th generalized Hamming weight of $\Cc$ is defined as
\begin{equation*}
    \mathrm{d}_r(\Cc)=\min\{\lvert\mathrm{supp}(\mathcal{D})\rvert:\mathcal{D}\text{ is a subcode of }\Cc\text{ such that }\dim_{\F_q}\mathcal{D}=r\}.
\end{equation*}
Generalized Hamming weights were introduced in 1991 by Wei~\cite{weiGHW} and have since been extensively studied. While they are well-understood for several families of codes (see, for example, \cite{pellikaanGHWRM,beelenGHWcartesian,munueraGHWhermitica,sanjoseGHWNT}), determining these weights for codes arising from higher-dimensional algebraic varieties remains a significant challenge. The case of $r=1$ is particularly important, and it corresponds to the so-called minimum distance of the code.

A generator matrix $G\subset \F_q^{k\times n}$ for a linear code $\Cc$ is a matrix whose rows form a basis for $\Cc$. We can consider the multiset $\Pi_G$ of the columns of $G$, viewed as projective points. This is natural since choosing a different representative for a point would correspond to multiplying a column by a nonzero scalar, which does not change the generalized Hamming weights of the linear code generated by the resulting generator matrix.  
In this setting, Van Tuyl and Tohǎneanu proposed the following conjecture concerning the minimum distance of a projective code whose associated point set $X = \Pi_G$ is a complete intersection; see also~\cite[Conjecture 4.29]{tohuaneanu2024commutative}.
\begin{conjecture}[\hspace{1pt}{\cite[Conjecture 4.9]{tohaneanu2013bounding}}]\label{conj}
 Let $\Cc\subseteq\F_q^n$ be a $k$-dimensional code with associated set of points $X$. If $X$ is a complete intersection with degrees $d_1,\dots,d_{k-1}$ such that $2\leq d_1\leq \dots\leq d_{k-1}$, then $\mathrm{d}_1(\Cc)\geq (d_1-1)d_2\cdots d_{k-1}$.
\end{conjecture}
Prior to this work, only partial cases of this conjecture had been established. The authors of the conjecture themselves proved it in $\mathbb{P}^2$~\cite[Theorem 4.10]{tohaneanu2013bounding}, as well as under a further geometric condition in~\cite[Corollary 4.8]{tohaneanu2013bounding}. Later, it was proven in~\cite[Corollary 4.2]{martinez2018minimum} for complete intersections whose initial ideal is also a complete intersection and in~\cite[Proposition 6.4]{cooper2020generalized} for the case $d_1=\cdots=d_{k-1}$. 

In this paper, we prove the conjecture. In fact, we study evaluation codes constructed by evaluating homogeneous forms at the points of a complete zero-dimensional intersection $X\subseteq\mathbb{P}(\bar\F_q)$ and we establish a new lower bound for the minimum distance of these codes for an arbitrary degree $d$, which yields the conjectured bound as a direct corollary when specialized to $d=1$. Furthermore, for the case of linear forms ($d=1$), we provide a lower bound for the $r$-th generalized Hamming weight in terms of the degrees $d_1,\dots,d_{k-1}$ of the complete intersection. Our results are built upon a refinement of the generalized Bézout bound for overdetermined systems.
\section{Bound for complete intersection codes}
We begin this section by proving a generalized Bézout bound for overdetermined systems, which serves as a key tool to derive bounds on the generalized Hamming weights subsequently. Although a foundational version of this bound appears in the work of Masser and Wüstholz~\cite[Theorem II, Chapter 2]{masser1983fields}, their analysis is restricted to the characteristic zero setting and to $p$-adic numbers. By adapting their approach to address positive characteristic and incorporating a technical refinement, we obtain a bound that is strictly sharper. Given that we could not find this refined statement explicitly stated in the literature, we also provide a direct, self-contained proof here accessible to readers with only a minimal background in algebraic geometry. 
Throughout, $\mathbb{K}$ denotes an algebraically closed field. For homogeneous polynomials $f_1,\dots,f_t \in \mathbb{K}[x_1,\dots,x_k]$ of positive degree, $V(f_1,\dots,f_t) \subseteq \mathbb{P}^{k-1}(\mathbb{K})$ denotes their common zero locus, and all dimensions and codimensions are the projective ones, so that $\dim V(J) = \dim \mathbb{K}[x_1,\dots,x_k]/J - 1$ for a homogeneous ideal $J \subseteq \mathbb{K}[x_1,\dots,x_k]$. 

We recall that if $f_1,\dots,f_t$, with $t < k$, is a regular sequence of homogeneous polynomials of positive degree, then $V(f_1,\dots,f_t)$ is nonempty of codimension $t$. A complete intersection of degrees $d_1,\dots,d_t$ is the subscheme of $\mathbb{P}^{k-1}(\mathbb{K})$ defined by a regular sequence of forms of degrees $d_1,\dots,d_t$. It is called reduced if the ideal generated by the regular sequence is radical, that is, if it coincides with the vanishing ideal of its zero locus. In particular, a zero-dimensional reduced complete intersection of degrees $d_1,\dots,d_{k-1}$ consists of exactly $d_1 \cdots d_{k-1}$ distinct points, by B\'ezout's theorem.
We begin by showing that the ideal generated by a system of homogeneous polynomials always contains a regular sequence of prescribed length and degrees.
\begin{proposition}
    Let $f_1,...,f_m$ be homogeneous polynomials of degree $0<d_1\leq\dots\leq d_m$ in $k$ variables such that the locus of common zeros has dimension $r$. Let $s\leq k-r-1$ and suppose that $f_1,\dots,f_s$ form a regular sequence. Then, in the homogeneous ideal $(f_1,...,f_m)$ there exists a regular sequence of homogeneous polynomials of length $k-r-1$ with degrees $d_1,\dots,d_s, d_{m-k+r+s+2},\dots,d_m$.
\end{proposition}
\begin{proof}
    For each $t\in[m]$, we denote by $J_t$ the ideal generated by $f_1,\dots, f_{m-t+1}$ and assume that $\dim(V(J_1))=r$. Since $V(J_1)$ is set-wise equal to $V(J_t)\cap V((f_{m-t+2},\dots,f_{m}))$, by Krull's height theorem, we get that
    $$\dim(V(J_t))\leq r+t-1.$$
    We construct by recursion a regular sequence of homogeneous polynomials $p_1,\dots, p_{k-r-1}$ with degrees $d_1,\dots,d_s, d_{m-k+r+s+2},\dots,d_m$ such that $p_i\in J_{k-r-i}$ for all $i\in[k-r-1]$. For $i\in[s]$ we set $p_i=f_i$. Assume now that we have a regular sequence $p_1,\dots,p_{\ell-1}$ satisfying the constraints above. In particular, the variety $V$ cut out by $p_1,\dots,p_{\ell-1}$ has codimension exactly $\ell-1$. Therefore, it is a complete intersection and every irreducible component of $V$ has dimension  $k-\ell$ by Krull's height theorem. We denote by $(J_{k-r-\ell})_{d_{m-k+r+\ell+1}}$ the vector space of homogeneous polynomials of degree $d_{m-k+r+\ell+1}$ in the ideal $J_{k-r-\ell}$. 
    
    Since $\dim(V(J_{k-r-\ell}))\leq k-\ell-1$ and $\dim(X)=k-\ell$ for any irreducible component $X$ of $V$, there exists a generator $f_u$ of $J_{k-r-\ell}$ that does not vanish on $X$. Recall that $J_{k-r-\ell}$ is generated by $f_1,\dots,f_{m-k+r+\ell+1}$, and we have $u\leq m-k+r+\ell+1$, i.e., $\deg(f_u)\leq d_{ m-k+r+\ell+1}$. Also note that $X$ not being empty implies that there exists a variable $x_j$ that does not vanish on $X$. Therefore, we have that $g=x_j^{d_{m-k+r+\ell+1}-\deg(f_u)}f_u\in(J_{k-r-\ell})_{d_{m-k+r+\ell+1}}$ is not identically zero on $X$ because $X$ is irreducible. Since $(J_{k-r-\ell})_{d_{m-k+r+\ell+1}}$ is a vector space over an infinite field and there is only a finite number of irreducible components of $V$, by prime avoidance there exists an element $p_\ell\in(J_{k-r-\ell})_{d_{m-k+r+\ell+1}}$ that does not vanish on any irreducible component of $V$, i.e., $\dim(V(p_1,\dots,p_\ell))=k-\ell-1$, which completes the recursion.
\end{proof}
The classical Bézout Theorem states that the number of points of $\mathbb{P}^{k-1}(\mathbb{K})$ contained in the intersection of $k-1$ hypersurfaces is either infinite or is bounded by the product of the degrees of the polynomials. This, together with the previous proposition, implies the following corollary.
\begin{corollary}[Generalized Bézout Theorem for overdetermined systems]\label{c:genBezout}
    Let $f_1,...,f_m$ be homogeneous polynomials of degree $d_1\leq\dots\leq d_m$ in $k$ variables such that the variety $V(f_1,\dots,f_m)$ has dimension $0$. Then,
    $$\lvert V(f_1,\dots,f_m)\rvert\leq d_1d_{m-k+3}\cdots d_m.$$
    Moreover, if $f_1,\dots,f_s$ form a regular sequence, then
    $$\lvert V(f_1,\dots,f_m)\rvert\leq d_1\cdots d_sd_{m-k+s+2}\cdots d_m.$$
\end{corollary}
We briefly recall the construction of projective evaluation codes. Let $\mathbb{F}_q[x_1,\dots,x_k]_d$ denote the vector space of homogeneous polynomials of degree $d$ over $\mathbb{F}_q$, together with the zero polynomial. Let $X= \{P_1,\dots,P_n\} \subseteq \mathbb{P}^{k-1}(\mathbb{F}_q)$ be a set of projective points, and fix a representative $v_i \in \mathbb{F}_q^k$ for each $P_i$. The projective evaluation code $\Cc_d(X)$ is defined as the image of the evaluation map $$\mathrm{ev} : \mathbb{F}_q[x_1,\dots,x_k]_d \longrightarrow \mathbb{F}_q^n, \quad f \longmapsto(f(v_1),\dots,f(v_n)).$$ A different choice of representatives rescales each coordinate of the code by a nonzero scalar, and hence yields a monomially equivalent code. In particular, the dimension and the generalized Hamming weights of $\Cc_d(X)$ do not depend on this choice, and with a slight abuse of language, we speak of evaluating at the points of $X$. Since a homogeneous polynomial vanishes at $v_i$ if and only if it vanishes at every representative of $P_i$, the minimum distance of $\Cc_d(X)$ is given by
\begin{equation}\label{eq:1}
    \mathrm{d}_1(\Cc_d(X))=n-\max\{\lvert X\cap V(f)\rvert\neq n:f\in \mathbb{F}_q[x_1, \dots, x_k]_d\}.
\end{equation}
Enlarging the field over which the forms are defined can only increase the maximum, so
$$\mathrm{d}_1(\Cc_d(X))\geq n-\max\{\lvert X\cap V(f)\rvert \neq n:f\in \bar\F_q[x_1, \dots, x_k]_d\}.$$
Moreover, the previous expression can be shown to be an equality, but the inequality is enough for our purposes. From this, by applying Corollary~\ref{c:genBezout}, we obtain the following bound.
\begin{theorem}\label{thm:mindisthigherform}
    Let $\Cc$ be a code obtained by evaluating the set of homogeneous forms of degree $d$ at the points of a zero-dimensional reduced complete intersection defined over $\F_q$ of degrees $2\leq d_1\leq d_2\leq\dots\leq d_{k-1}$. Then,
    $$\mathrm{d}_1(\Cc)\geq (d_1-\min\{d,d_1\})d_{2}\cdots d_{k-1}.$$
\end{theorem}
\begin{proof}
    Since the set of evaluation points forms a reduced complete intersection, we have $n=d_1d_2\cdots d_{k-1}$. Moreover, let $f_1,\dots, f_{k-1}$ be the polynomials that define the complete intersection. Then, for every homogeneous form $f$ of degree $d$, either the system $f,f_1,\dots,f_{k-1}$ has no solutions or, if its zero locus is a zero-dimensional variety, it follows from Corollary \ref{c:genBezout}, that its cardinality is bounded from above by $\min\{d,d_1\}d_{2}\cdots d_{k-1}$. This, together with the discussion preceding the theorem, completes the proof.
\end{proof}
\begin{remark}
    When $d\geq d_1$, the bound is trivial. However, when $d<d_1$, it is significantly better than the one obtained by Gold, Little, and Schenck in~\cite[Theorem 3.2]{gold2005cayley}.
\end{remark}

A code whose dual has no codewords of weight $1$ or $2$ is known as a projective code. It is well known that every $[n,k]$ linear projective code can be realized as a projective evaluation code, where the codewords are generated by evaluating forms of degree $1$ on the set of columns of a generator matrix. From this perspective, Equation~\eqref{eq:1} can be extended to generalized Hamming weights as
$$\mathrm{d}_r(\Cc)=n-\text{hyp}_r(\Cc),$$
where $\text{hyp}_r(\Cc)$ denotes the largest number of columns of $G$ contained in a $k-r$ dimensional vector space. Notice that $\text{hyp}_r(\Cc)$ does not depend on the choice of the generator matrix $G$. We refer to~\cite[Chapters 4 and 5]{tohuaneanu2024commutative} for more details. In this context, we are able to extend the statement of Theorem~\ref{thm:mindisthigherform} to generalized Hamming weights. We remark that Conjecture~\ref{conj} follows from this result by setting $r=1$. Also note that whenever $X$ is reduced, the code $\Cc$ does not have two linearly dependent columns (i.e., the code is nondegenerate), and then $\mathrm{d}_k(\Cc)=n=d_1\cdots d_{k-1}$.
\begin{theorem}\label{thm:generalizedweightshyper}
    Let $\Cc$ be a $k$-dimensional linear code obtained by evaluating the set of homogeneous linear forms at the points of a zero-dimensional reduced complete intersection defined over $\F_q$ of degrees $2\leq d_1\leq d_2\leq\dots\leq d_{k-1}$. Then, for $1\leq r \leq k-1$, we have
    $$\mathrm{d}_r(\Cc)\geq (d_1\cdots d_r-1)d_{r+1}\cdots d_{k-1}.$$
\end{theorem}
\begin{proof}
    Since the set of evaluation points projectively forms a reduced complete intersection, we have $n=d_1d_2\cdots d_{k-1}$. Moreover, we know that  
    \begin{equation*}
        \begin{split}
            \mathrm{d}_r(\Cc)=n-\max\{\lvert X\cap V(h_1,\dots,h_r)\rvert:&h_1,\dots,h_r\in \mathbb{F}_q[x_1, \dots, x_k]_1\\&\text{ and }\dim(V(h_1,\dots,h_r))=k-1-r\}
        \end{split}
    \end{equation*}
    This follows from the fact that the $r$-th generalized weight is the minimum cardinality of the support of a subcode of dimension $r$. This is equivalent to considering the maximum number of columns that lie on $r$ linearly independent hyperplanes simultaneously. Finally, note that $r$ linearly independent hyperplanes form a regular sequence. So, to conclude, it suffices to apply Corollary~\ref{c:genBezout} to the system $h_1,\dots,h_r,f_1,\dots,f_{k-1}$. 
\end{proof}

\begin{remark}
    Theorem~\ref{thm:mindisthigherform} cannot be extended to generalized Hamming weights with the same approach of Theorem~\ref{thm:generalizedweightshyper} since the linear independence of homogeneous forms of the same degree does not in general imply the regularity of the sequence.
\end{remark}

\section{A more general conjecture}
Let $A_i\subseteq\mathbb{F}_q$ be a set of cardinality $d_i$, for $1\leq i \leq k-1$, and consider the projective Cartesian set 
\[ 
X_{\text{Car}} =\{[1:a_1:\cdots:a_{k-1}]\in\mathbb{P}^{k-1}(\mathbb{F}_q): a_i\in A_i\}. 
\] 
Thus, $X_{\mathrm{Car}}$ is a reduced complete intersection of degrees $d_1,\dots,d_{k-1}$. The generalized Hamming weights $\mathrm{d}_r\bigl(\mathcal{C}_d(X_{\mathrm{Car}}))$ of $\mathcal{C}_d(X_{\mathrm{Car}})$ are known \cite{beelenGHWcartesian}, and we conjecture that projective Cartesian codes have the smallest generalized Hamming weights among evaluation codes supported on reduced complete intersections with the same degrees. 

\begin{conjecture}\label{conj:ghws}
Let $X\subseteq\mathbb{P}^{k-1}(\mathbb{F}_q)$ be a zero-dimensional reduced complete intersection of degrees $2\leq d_1\leq\cdots\leq d_{k-1}$, and let $X_{\mathrm{Car}}$ be as above. Then, for every $d\geq 1$ and every $1\leq r\leq\dim\mathcal{C}_d(X)$, 
\[ 
\mathrm{d}_r\bigl(\mathcal{C}_d(X)\bigr) \geq \mathrm{d}_r\bigl(\mathcal{C}_d(X_{\mathrm{Car}})\bigr). 
\] 
\end{conjecture} 

For the minimum distance, this conjecture holds when $X\subseteq\mathbb{P}^2$. Indeed, for $d\geq d_1-1$, the bound of Gold, Little, and Schenck \cite[Theorem 3.2]{gold2005cayley} agrees with the minimum distance of the corresponding Cartesian code. For $d<d_1$, the conjecture holds for any $k$ by Theorem \ref{thm:mindisthigherform}, and for $d=1$, the conjecture is true by Theorem \ref{thm:generalizedweightshyper}. Finally, for $r=1$, Conjecture \ref{conj:ghws} is equivalent to CB12 in \cite{eisenbud_green_harris} for reduced complete intersections.

\section*{Disclosure of AI usage}
We acknowledge the use of ChatGPT Pro 5.5 (OpenAI) in the early stages of this work. The AI suggested an initial proof strategy for the conjecture involving a weaker version of Proposition 2.1 and a projection argument. The final proof was independently simplified and extended to generalized Hamming weights and forms of larger degree.
\bibliographystyle{abbrv}
\bibliography{biblio.bib}
\end{document}